\documentclass[]{amsart}
\usepackage[utf8]{inputenc}
\usepackage[mathscr]{eucal} 
\usepackage{stmaryrd} 
\usepackage{amsmath,amsthm,amsfonts,amssymb,amscd} 
\usepackage[dvipsnames]{xcolor} 
\usepackage{xspace} 
\usepackage{fancyhdr} 
\usepackage{graphicx} 
\usepackage{listings} 
\usepackage[]{hyperref} 
\usepackage{enumitem} 
\usepackage{relsize} 
\usepackage{tikz-cd} 
\usepackage{mdwlist} 
\usepackage{multicol} 
\usepackage{float} 
\usepackage{adjustbox} 
\usepackage{tikz} 
\usepackage[noabbrev,nameinlink]{cleveref} 
\Crefformat{section}{#2\S#1#3}
\usepackage{setspace} 
\usepackage{bbm} 
\usepackage{mathtools}
\usepackage[all]{xy}
\usetikzlibrary{matrix, calc, arrows}

\hypersetup{
	colorlinks=true, 
	linktoc=all,     
	linkcolor=Brown,citecolor=Brown,urlcolor=MidnightBlue,  
}
\setlist[enumerate]{label=(\alph*)} 
\usetikzlibrary{graphs,decorations.pathmorphing,decorations.markings}
\tikzcdset{scale cd/.style={every label/.append style={scale=#1},
		cells={nodes={scale=#1}}}}

\newcommand{\Hom}{\operatorname{Hom}}

\newcommand{\inv}{^{-1}}

\newcommand{\op}{^{\operatorname{op}}}
\newcommand{\Res}{\operatorname{Res}}
\newcommand{\Ind}{\operatorname{Ind}}
\newcommand{\Inf}{\operatorname{Inf}}
\newcommand{\Def}{\operatorname{Def}}
\newcommand{\Iso}{\operatorname{Iso}}
\newcommand{\id}{\operatorname{id}}

\newcommand{\Pic}{\operatorname{Pic}}

\newcommand{\on}[1]{\operatorname{#1}}

\newcommand{\calC}{\mathcal{C}}
\newcommand{\calD}{\mathcal{D}}

\newcommand{\calF}{\mathcal{F}}

\newcommand{\catK}{\mathscr{K}}

\newcommand{\bbF}{\mathbb{F}}

\newcommand{\bbQ}{\mathbb{Q}}

\newcommand{\bbZ}{\mathbb{Z}}

\makeatletter
\newcommand*{\doublerightarrow}[2]{\mathrel{
		\settowidth{\@tempdima}{$\scriptstyle#1$}
		\settowidth{\@tempdimb}{$\scriptstyle#2$}
		\ifdim\@tempdimb>\@tempdima \@tempdima=\@tempdimb\fi
		\mathop{\vcenter{
				\offinterlineskip\ialign{\hbox to\dimexpr\@tempdima+1em{##}\cr
					\rightarrowfill\cr\noalign{\kern.5ex}
					\rightarrowfill\cr}}}\limits^{\!#1}_{\!#2}}}
\newcommand*{\triplerightarrow}[1]{\mathrel{
		\settowidth{\@tempdima}{$\scriptstyle#1$}
		\mathop{\vcenter{
				\offinterlineskip\ialign{\hbox to\dimexpr\@tempdima+1em{##}\cr
					\rightarrowfill\cr\noalign{\kern.5ex}
					\rightarrowfill\cr\noalign{\kern.5ex}
					\rightarrowfill\cr}}}\limits^{\!#1}}}
\makeatother

\newcommand{\RO}{\on{RO}}
\newcommand{\CF}{\on{CF}}
\newcommand{\sgn}{\on{sgn}}

\newtheorem{theorem}{Theorem}[section]

\newtheorem{proposition}[theorem]{Proposition}
\newtheorem{corollary}[theorem]{Corollary}

\newtheorem*{theorem*}{Theorem}

\theoremstyle{remark}

\newtheorem{remark}[theorem]{Remark}
\newtheorem{example}[theorem]{Example}

\newtheorem{notation}[theorem]{Notation}

\theoremstyle{definition}
\newtheorem{definition}[theorem]{Definition}

\makeatletter
\newcommand{\xhookdoubleheadrightarrow}[2][]{%
	\lhook\joinrel
	\ext@arrow 0359\rightarrowfill@ {#1}{#2}%
	\mathrel{\mspace{-15mu}}\rightarrow
}
\makeatother

\newtheorem{innercustomthm}{Theorem}
\newenvironment{customthm}[1]
{\renewcommand\theinnercustomthm{#1}\innercustomthm}
{\endinnercustomthm}

\AddToHook{env/corollary/begin}{\crefalias{theorem}{corollary}}
\AddToHook{env/proposition/begin}{\crefalias{theorem}{proposition}}
\AddToHook{env/lemma/begin}{\crefalias{theorem}{lemma}}
\AddToHook{env/definition/begin}{\crefalias{theorem}{definition}}
\AddToHook{env/example/begin}{\crefalias{theorem}{example}}
\AddToHook{env/remark/begin}{\crefalias{theorem}{remark}}
\AddToHook{env/observation/begin}{\crefalias{theorem}{observation}}
\AddToHook{env/construction/begin}{\crefalias{theorem}{construction}}

\begin{document}
	\title{The fusion-stable tom Dieck homomorphism}
	\author{Sam K. Miller}
	\address{Department of Mathematics, University of Georgia, Athens GA 30602, United States of America} 
	\email{sam.miller@uga.edu} 
	\subjclass[2020]{19A22 } 
	\keywords{tom Dieck homomorphism, fusion system, real representation, characteristic idempotent, Borel--Smith function, endotrivial complex} 
	\begin{abstract}
		Tornehave and Yal\c{c}in proved that the tom Dieck homomorphism, which sends a virtual real representation to a unit of the Burnside ring, is surjective for any $p$-group $S$. We prove that this homomorphism, and the sign homomorphism it factors through, remain surjective when restricted to fusion-stable subgroups associated to a saturated fusion system on $S$. As a corollary, we close the main question posed by Mazza--Miller by showing that given a field $k$ of positive characteristic, the Lefschetz homomorphism from the Picard group of the bounded homotopy category of $p$-permutation modules to the unit group of its Grothendieck ring is surjective for all finite groups if and only if $k = \bbF_2$. 
	\end{abstract}
	
	\maketitle

	\section*{Introduction}
	
	Let $p$ be any prime and $S$ be a finite $p$-group. We consider some classical homomorphisms arising in equivariant topology and representation theory: the \emph{dimension homomorphism}: 
	\[\dim_S\colon \on{RO}(S) \to \on{CF}_b(S),\quad V \mapsto \big(P \mapsto \dim V^P\big)\] and the \emph{sign homomorphism} 
	\[\on{sgn}_S\colon \on{CF}_b(S)\to A(S)^\times,\quad f \mapsto \big(P \mapsto (-1)^{f(P)}\big).\] The composition of these homomorphisms is the so-called \emph{tom Dieck homomorphism} $\Theta_S$. Here, $\on{RO}(S)$ is the Grothendieck group of real representations of $S$, $\on{CF}_b(S)$ is the group of \emph{Borel-Smith superclass functions} of $S$, $A(S)$ denotes the \emph{Burnside ring} of $S$, and we regard $A(S)^\times$ as a subgroup of $\on{CF}(S)^\times$ via the injective mark homomorphism 
	\[m_S\colon B(S) \to \on{CF}(S), \quad X \mapsto \big(P \mapsto |X^P|\big).\] 
	
	Both the dimension and sign homomorphisms are well-defined and surjective.  Dotzel--Hamrick \cite{DH81} verified surjectivity of the dimension homomorphism, and Tornehave \cite{Tor84} verified in an unpublished manuscript that the tom Dieck homomorphism is surjective via a topological argument. Yal\c{c}in \cite{Y05} later verified Tornehave's result algebraically via the theory of rational $p$-biset functors.  
	
	If we remove the restriction that $S$ is a $p$-group, then these homomorphisms are known to no longer be surjective, see e.g. \cite[Remark 5.3]{Bar11}. However, one can consider the question $p$-locally instead by considering the fusion system $\calF_S(G)$ associated to a finite group $G$ with $p$-Sylow $S$. More generally, we may take $\calF$ to be any fusion system on a $p$-group $S$, and one can consider the subgroups of $\calF$-stable elements $\on{RO}(\calF)$, $\on{CF}_b(\calF)$, and $A(\calF)^\times$, respectively. We obtain corresponding $\calF$-stable group homomorphisms:
	\[\dim_\calF\colon \on{RO}(\calF) \to \on{CF}_b(\calF) \text{,  } \sgn_\calF\colon \on{CF}_b(\calF) \to A(\calF)^\times \text{,  and  } \Theta_\calF := \sgn_\calF\circ \dim_\calF.\]
	Apriori it is not clear these homomorphisms remain surjective. In fact, they do not. Reeh--Yal\c{c}in previously considered this question for the dimension homomorphism, and verified that its image consists of the $\calF$-stable elements satisfying Bauer's \emph{Artin condition} \cite{RY18}. 
	
	We show in this short note that on the other hand, the $\calF$-restricted tom Dieck homomorphism and $\calF$-restricted sign homomorphism remain surjective. 
	
	\begin{customthm}{A}\label{thm:A}(\Cref{thm:main})
		Let $\calF$ be a saturated fusion system over a $p$-group $S$. The $\calF$-restricted sign homomorphism \[\sgn_\calF\colon \CF_b(\calF) \to A(\calF)^\times\] and the $\calF$-restricted tom Dieck homomorphism \[\Theta_\calF\colon \RO(\calF) \to A(\calF)^\times\] are well-defined and surjective.
	\end{customthm}
	
	\Cref{thm:A} answers a question posed by Mazza and the author, \cite[Remark 4.2]{MM26}. It also enables us to, in a sense, close the main question of \cite{MM26} regarding the Euler characteristic of an endotrivial complex, i.e., an invertible objects in the bounded homotopy category of $p$-permutation modules over the field $\bbF_2$. See \Cref{rmk:endotriv_overview} for details.
	
	\begin{customthm}{B}(\Cref{cor:thmB})
		Let $G$ be a finite group. The Lefschetz homomorphism \[\Lambda\colon \Pic(\catK(G;\bbF_2)) \to O(T(\bbF_2G)), \quad C \mapsto \sum_{i \in \bbZ} (-1)^i[C_i]\] is surjective. In particular, the following holds: if $k$ is field of characteristic $p$, then the Lefschetz homomorphism is surjective for every finite group if and only if $k = \bbF_2$.
	\end{customthm}
	
	There are two key mechanisms which go into the proof of \Cref{thm:A}. First, we use the existence of so-called \emph{characteristic idempotents} for saturated fusion systems. These are bifree idempotents in the $p$-localized \emph{biset Burnside ring}. The idea of relating fusion systems to double Burnside rings and characteristic elements reportedly goes back to Linckelmann--Webb in an unpublished manuscript, and the existence and uniqueness of characteristic idempotents of saturated fusion systems is due to Ragnarsson \cite{Rag06} and Ragnarsson--Stancu \cite{RS13}. 
	Second, $\on{RO}(-), \on{CF}_b(-),$ and $\on{B}(-)^\times$ are all \emph{$p$-biset functors} \cite{Bou10}, therefore there is a natural action of the $p$-localized double Burnside ring on $p$-localizations of all associated groups. In our setting, the action of a characteristic idempotent $p$-locally $\calF$-stabilizes an element of $\on{RO}(G)_{(p)}$, $\on{CF}_b(G)_{(p)}$, and $A(G)^\times_{(p)}$ by results of Rees--Yal\c{c}in \cite{RY18}, and one can lift the stabilization to the global level here. 
	
	The paper is organized as follows: Section 1 concerns generalities on biset functors and the groups and homomorphisms in question for this apper, Section 2 briefly reviews fusion systems and characteristic idempotents, and Section 3 proves the main results. 
	
	\subsection*{Acknowledgments and statement on generative AI use} I thank Robert Boltje for first introducing me to biset functors, as well as plenty of other topics arising in this paper. Moreover, he suggested during my Ph.D. that I consider looking at characteristic idempotents to answer the questions here. 
	
	I recently was testing the capabilities of a generative AI model, and after some prodding, it also attempted to use characteristic classes to answer this question, which reignited my motivation to close the problem. However, the proofs and paper are produced entirely by me, and I take full responsibility for their content. 
	
	\section{Biset functors}
	
	We first review basics of biset functors, following Bouc \cite{Bou10}. We refer the reader to \cite[Chapters 2, 3]{Bou10} for details. 
	
	\begin{definition}
		Let $G,H, K$ be finite groups. We let $A(G)$ denote the \emph{Burnside ring}, i.e., the Grothendieck ring on the category of left $G$-sets. We let $A(H,G)$ denote the \emph{biset Burnside group}, which is nothing more than $A(H \times G\op)$. We have a composition law \[\circ\colon A(K,H) \times A(H,G) \to A(K,G)\] such that $[V]\circ [U] = [V\times_H U]$; this endows $A(G,G)$ with noncommutative ring structure with unit the $(G,G)$-biset $G$. Given a $(H,G)$-biset $V$, its \emph{opposite $(G,H)$-biset} $V\op$ is $V$ as a set and has $(G,H)$-action given by $g\cdot v\cdot h := h\inv vg\inv$.
		
		We have an injective ring homomorphism, the \emph{mark homomorphism} \[m\colon A(G) \to \on{CF}(G),\quad [X] \mapsto \big(H \mapsto |X^H|\big).\] The values of the resulting superclass function are referred to as the \emph{marks} of the (virtual) $G$-set. Since the mark homomorphism is injective, elements of $A(G)$ or $A(H,G)$ may often be described instead via their marks instead, i.e., as superclass functions. The mark homomorphism is in general non-surjective, so there is no guarantee a superclass function corresponds to a virtual $G$-set. 
	\end{definition}
	
	\begin{remark}
		We list five fundamental examples of bisets.
		\begin{itemize}
			\item If $H \leq G$, then $G$ is the \emph{restriction} $(H,G)$-biset and is denoted by $\Res^G_H$;
			\item Similarly, $G$ is the \emph{induction} $(G,H)$-biset and is denoted by $\Ind^G_H$;
			\item If $N \trianglelefteq G$, then $G/N$ is the \emph{inflation} $(G,G/N)$-biset and is denoted by $\Inf^G_{G/N}$;
			\item Similarly, $G/N$ is the \emph{deflation} $(G/N, G)$-biset and is denoted by $\Def^G_{G/N}$;
			\item If $\varphi\colon G\to H$ is a group isomorphism, then $H$ is the \emph{isomorphism} $(H,G)$-biset and is denoted by $\Iso(f)$.
		\end{itemize}
		In fact, every transitive biset can be written as a composition of these five bisets. See \cite[Lemma 2.3.26]{Bou10} for details. Moreover, any transitive \emph{bifree} (left and right free) biset can be written as a composition of only induction, restriction, and isomorphism.  
	\end{remark}
	
	\begin{example}\label{ex:twisted_diag}
		Let $G,H$ be finite groups and let $K \leq G$ and let $\varphi\colon K \to H$ be an injective group homomorphism. Set $\Delta(K, \varphi) := \{(\varphi(k),k) \mid k \in K\}$, a \emph{twisted diagonal} subgroup of $H \times G$. We denote the transitive, bifree $(H,G)$-biset $H \times G/\Delta_\varphi(K)$ by $[K,\varphi]^H_G$. We have an equality \[[K,\varphi]^H_G = \Ind^H_{\varphi(K)}\circ \Iso(\varphi) \circ \Res^G_K.\]
	\end{example}
	
	\begin{definition}
		\begin{enumerate}
			\item The \emph{biset category} $\calC$ of finite groups is the category defined as follows:
			\begin{itemize}
				\item The objects of $\calC$ are finite groups;
				\item If $G,H$ are finite groups, then $\Hom_\calC(G,H) = A(H,G)$;
				\item If $G,H,K$ are finite groups, then composition $v\circ u$ of morphisms $v \in \Hom_\calC(H,K)$ and $u\in \Hom_\calC(G,H)$ is equal to $v \circ u$;
				\item The identity morphism of $G \in \calC$ is the virtual biset $[G]$.
			\end{itemize}
			Given a commutative ring $R$, there is a corresponding $R$-linear biset category $R\calC := R \otimes_\bbZ \calC$. 
			\item Let $\calD \subseteq \calC$ be a preadditive subcategory. A \emph{biset functor} on $\calD$ is with values in $\mathsf{Mod}(R)$ is an $R$-linear functor $R\calD \to \mathsf{Mod}(R)$. Given a $R$-linear biset functor $F$ and biset $U\in A(H,G)$, we write $T_U\colon F(G) \to F(H)$ to denote the corresponding $R$-module homomorphism. A morphism of $R$-linear biset functors is simply a natural transformation. We say a morphism is \emph{injective} (resp. \emph{surjective}) if it is pointwise injective (resp. surjective). 
			
			\item We let $\calC_p$ denote the full subcategory of $\calC$ consisting of all $p$-groups, and call a biset functor on $\calC_p$ a \emph{$p$-biset functor.} 
			
		\end{enumerate}
	\end{definition}
	
	The theory of \emph{rational} $p$-biset functors, those satisfying conditions similar to those of the rational representation ring $p$-biset functor $R_\bbQ(-)$, as led to significant developments in group representation theory, such as the determination of the Dade group of a $p$-group \cite{Bou06} and the unit group of the Burnside ring of a $p$-group \cite{Y05, Bou07}. 
	
	\begin{definition}
		If $F$ is a $R$-linear biset functor on $\calC$ or $\calC_p$, the \emph{dual biset functor} $F^*$is the biset functor defined by $F^*(G) = \Hom_\bbZ(F(G), R)$ for an appropriate finite group $G$, and $F^*(U) = {}^tF(U\op)$ for any $(H,G)$-biset $U$, where ${}^tF(U\op)$ denotes the transposed map of $F(U\op)$. 
	\end{definition}
	
	\begin{remark}
		Because every transitive biset can be written as a composition of the five fundamental bisets (induction, restriction, inflation, deflation, isomorphism), specifying a $R$-linear biset functor on $\calC$ amounts to specifying for every finite group $G \in \calC$ a $R$-module, with $R$-module homomorphisms corresponding to induction, restriction, inflation, deflation, and isomorphism, satisfying certain compatibility conditions listed in \cite[Subsection 1.1.3]{Bou10}. 
	\end{remark}
	
	\subsection{The main examples} We describe a few biset functors, the main figures of the paper. 
	
	\begin{definition}
		We denote by $\on{RO}(G)$ the real representation group of $G$. Then the assignment $G \mapsto \on{RO}(G)$ induces a biset functor $\RO(-)$, with the induction biset corresponding to usual (co)induction of finite groups, the restriction biset corresponding to restriction of representations, the inflation biset corresponding to inflation of representations, the deflation biset corresponding to taking fixed points of representations, and the isomorphism biset defined in the obvious way. See \cite[Section 3.d]{Yos90}.
	\end{definition}
	
	\begin{definition}  \label{def:class_fns_borel_smith_biset}
		\begin{enumerate}
			\item The assignment $G \mapsto A(G)$ is a biset functor $A(-)$. Any $U \in A(H,G)$ induces a map $T_U\colon A(G) \to A(H)$ given by $X \mapsto U \circ X $ for any $X \in A(G)$. One may verify that the induction biset corresponds to (additive) induction, the restriction biset corresponds to restriction, the inflation biset corresponds to inflation, and the deflation biset corresponds to taking orbits. 
			\item The dual biset functor $A(-)^*$ satisfies that any $U \in A(H,G)$ induces a morphism $T_U \in \Hom_\calC(G,H)$ defined as follows. If $f \in A(G)^*$, we have \[T_U(f)(X) := f(U\op \circ X)\] for any $X \in A(H)$. 
			\item Let $R$ be a commutative ring. We let $\CF(G,R)$ denote the ring of $R$-valued \emph{superclass functions}. If $R= \bbZ$, we write $\CF(G)$ for short.
			
			The biset $A(-)^*$ transports its structure to the biset functor $\on{CF}(-)$ in the following way. An element of $A(G)^*$ is determined by its values on the transitive $G$-sets $G/H$. It follows that the assignment \[\on{CF}(G) \to A(G)^*,\quad f \mapsto \big([G/H] \mapsto f(H)\big)\] is a ring isomorphism. 
			The biset structure on $\on{CF}(-)$ which upgrades this to a biset functor isomorphism is as follows:
			given a $(H,G)$-biset $U$ inducing $T_U \in \Hom_\calC(G,H)$ and $f \in \on{CF}(G)$, we have \[T_U(f)(H) = \sum_{u \in H\backslash U/G}f( H^u),\] where $H^u $ is the subgroup of $G$ defined by \[H^u := \{g \in G \mid \exists h\in H \text{ such that } hu = ug\}.\] See \cite[3.4]{BoYa07} for details. 
			More generally, $\CF(-, R)$ has the structure of a $R$-linear biset functor via the same biset action.
			\item Let $S$ be a $p$-group and let $\on{CF}_b(S)$ denote the subgroup of $\on{CF}(S)$ satisfying the following \emph{Borel--Smith} conditions:
			\begin{itemize}
				\item If $p$ is odd, then for any subquotient $Q/P$ of $G$ of order $p$, $f(Q) \equiv f(P) \mod 2$.
				\item If $p = 2$, then for any sequence of subgroups $P \trianglelefteq Q \trianglelefteq R \leq N_S(P)$, with $[Q:P] = 2$, $f(P) \equiv f(Q) \mod 2$ if $R/P$ is cyclic of order 4 and $f(P) \equiv f(Q) \mod 4$ is $R/P$ is quaternion of order 8.
				\item For any elementary abelian subquotient $Q/P$ of $S$ of rank 2, the equality \[f(P) - f(Q) = \sum_{P < X < Q} \big(f(X) - f(Q)\big)\] holds.
			\end{itemize}
			A function $f \in \on{CF}_b(S)$ is called a \emph{Borel--Smith} function. In \cite{BoYa07}, it is proven that $\on{CF}_b(-)$ is a $\calC_p$-biset subfunctor of $\on{CF}(-)$.
			
			Borel--Smith are intimately related to the study of representation spheres (see \cite{DH81, tD87}), endopermutation modules (see \cite{Bou06}), and endotrivial complexes (see \cite{M24c, GM26} and \Cref{rmk:endotriv_overview}). 
		\end{enumerate}
	\end{definition}
	
	\begin{definition}
		The assignment $G \mapsto A(G)^\times$ defines a $\calC_p$-biset functor $A(-)^\times$. The biset operation differs from that on $A(G)$; it is given by so-called \emph{generalized tensor induction}. In this setting, the restriction biset and inflation biset correspond to usual restriction and inflation, but induction corresponds to tensor induction, and deflation corresponds to the fixed points operation. See \cite{Bou07} or \cite[Section 11.2]{Bou10}. 
		
		In general, determining $A(G)^\times$ for non-$p$-groups is a famously difficult problem. It is classically known by an argument of tom Dieck that the statement ``if $G$ is solvable, then $|A(G)^\times| = 2$'' is equivalent to the Feit--Thompson odd order theorem, see \cite[Theorem 11.2.4]{Bou10}.
	\end{definition}
	
	\subsection{Morphisms of biset functors}
	
	We return to the morphisms mentioned in the introduction. Let $S$ be a $p$-group, then we have the \emph{dimension homomorphism}: 
	\[\dim_S\colon \on{RO}(S) \to \on{CF}_b(S),\quad V \mapsto \big(P \mapsto \dim V^P\big)\] and the \emph{sign homomorphism} 
	\[\on{sgn}_S\colon \on{CF}_b(S)\to A(S)^\times,\quad f \mapsto \big(P \mapsto (-1)^{f(P)}\big).\]    
	The \emph{tom Dieck homomorphism} $\Theta_S\colon \on{RO}(S) \to A(S)^\times$ is the composition:
	\[\Theta_S = \on{sgn}_G\circ\dim_G\colon V \mapsto \big(S \mapsto (-1)^{\dim V^S}\big).\]
	
	The dimension homomorphism and tom Dieck homomorphisms are known to be morphism of $p$-biset functors, and are surjective. 
	
	\begin{proposition}
		The dimension homomorphism $\dim_S$ is surjective, and the collection of dimension homomorphisms for every $p$-group induces a morphism of $\bbZ$-linear $p$-biset functors. 
	\end{proposition}
	\begin{proof}
		Surjectivity of $\dim_S$ is proven by Dotzel--Hamrick \cite{DH81} (see also tom Dieck \cite{tD87}), and that $\dim$ determines a $p$-biset functor is shown in \cite{BoYa07} in the remark after Theorem 3.3.
		
	\end{proof}
	
	\begin{proposition}
		The tom Dieck homomorphism $\Theta_S$ is surjective, and the collection of tom Dieck homomorphisms for every $p$-group induces a morphism of $\bbZ$-linear $p$-biset functors. 
	\end{proposition}
	\begin{proof}
		Surjectivity of $\Theta_S$ is proven by Tornehave \cite{Tor84} and Yal\c{c}in \cite{Y05}, and that $\Theta$ determines a biset functor is shown by Yoshida \cite[Lemma 3.5]{Yos90}.
		
	\end{proof}
	
	We show the sign homomorphism is also a morphism of biset functors and characterize it as nothing more than base change. However, biset functoriality can also be shown directly via diagram chase from biset functoriality of $\Theta$ and $\dim$, and surjectivity of $\dim$. 
	Since $\on{CF}(S)^\times$ is nothing more than superclass functions with values in $\pm 1$, we clearly have a group isomorphism 
	\[ \on{CF}(S,\bbF_2) \to \on{CF}(S)^\times, \quad f \mapsto \big(P \mapsto (-1)^{f(P)}\big).\] As we view $A(S)^\times$ as a subgroup of $\on{CF}(S)^\times$ via the mark homomorphism $m_S$, we obtain a commutative diagram as follows. 
	
	\begin{figure}[H]
		\centering
		\begin{tikzcd}
			&  \on{CF}(G, \bbF_2) \ar[dd, "\cong"] \\
			A(G)^\times \ar[ur, hookrightarrow, "r_G"] \ar[dr, hookrightarrow, "m_G"] \\
			& \on{CF}(G)^\times
		\end{tikzcd}
	\end{figure}
	
	\begin{proposition}\label{prop:identification_of_unit_biset}
		Let $S$ be a $p$-group. The inclusions $r_S\colon A(S)^\times \hookrightarrow  \mathrm{CF}(S, \bbF_2)$ induce an inclusion of $\bbF_2$-linear $p$-biset functors $r\colon A(-)^\times \hookrightarrow  \mathrm{CF}(-, \bbF_2)$.
	\end{proposition}
	\begin{proof}
		If $p$ is odd, then $A(S)^\times = \{\pm[S/S]\}$ and the result is straightforward. Otherwise if $p=2$, this follows from \cite[Proposition 11.2.32]{Bou10} after identifying $ \mathrm{CF}(-, \bbF_2)$ as the $\bbF_2$-dual biset functor $\Hom_\bbZ(A(-), \bbF_2)$; \cite[Proposition 11.2.32]{Bou10} asserts an inclusion of $\bbF_2$-linear $2$-biset functors \[\iota\colon A(-)^\times \to \Hom_\bbZ(A(-), \bbF_2).\] 
		
	\end{proof}
	
	In other words, $A(-)^\times$ is a $\bbF_2$-linear $2$-biset subfunctor of $\CF(-, \bbF_2)\cong \bbF_2 \otimes_\bbZ \CF(-)$. Note these are $\bbZ$-linear biset functors as well via the base change $\bbZ \to \bbF_2.$
	
	\begin{proposition}\label{prop:sgn_is_biset_mor}
		Let $G$ be a $p$-group. The composition \[r_G \circ \on{sgn}_G\colon \CF_b(G) \twoheadrightarrow A(G)^\times \hookrightarrow \mathrm{CF}(G, \bbF_2)\] identifies with the base change \[\bbF_2 \otimes_\bbZ- \colon \CF_b(G)\to \bbF_2\otimes_\bbZ\mathrm{CF}(G) \cong \CF(G,\bbF_2),\] a $\bbZ$-linear morphism of $p$-biset functors. Consequently, the collection of group homomorphisms $\on{sgn}_G$ collectively induce a morphism between $\bbZ$-linear $p$-biset functors \[\on{sgn}\colon \CF_b(-) \twoheadrightarrow A(-)^\times.\]
	\end{proposition}
	\begin{proof}
		The first statement is a routine verification which we leave to the reader. For the second, since $2\bbZ \subset \bbZ$ is flat, $\bbF_2 \otimes_\bbZ - \colon \on{CF}_b(-) \to \bbF_2\otimes_\bbZ \mathrm{CF}(-)$ intertwines with the biset structure of $\CF(-)$ (see \Cref{def:class_fns_borel_smith_biset}) and therefore is a morphism of $p$-biset functors. As $r$ is a $p$-biset functor by \Cref{prop:identification_of_unit_biset} and injective, we conclude $\sgn$ is compatible with the biset operations and hence also is a morphism of $\bbZ$-linear $p$-biset functors. 
		
	\end{proof}

	\section{Fusion systems and characteristic idempotents}
	
	For this paper, we will essentially black-box the technical details of fusion systems. The key point is the existence of so-called characteristic idempotents. For a comprehensive overview of fusion systems, see \cite{AKO11}.
	
	\begin{definition}
		Let $S$ be a $p$-group. A \emph{fusion system} $\calF$ on $S$ is a category whose objects are the subgroups of $S$ and whose morphisms between subgroups $P$ and $Q$ of $S$ is a set $\Hom_\calF(P,Q)$ of injective group homomorphisms from $P$ to $Q$ with the following properties:
		\begin{enumerate}
			\item If ${}^gP \leq Q$, then $c_g\colon P \to {}^gP \in \Hom_\calF(P,Q)$;
			\item For any $\varphi\in \Hom_\calF(P,Q)$, the induced homomorphism $P \cong \varphi(P)$ and its inverse are morphisms in $\calF;$
			\item The composition of morphisms in $\calF$ is the usual composition of group homomorphisms.
		\end{enumerate}
		
		Given a subgroup $P \leq S$, we let $[P]_\calF$ denote the set of subgroups isomorphic to $P$ in $\calF$. We omit the rather technical definition of a \emph{saturated} fusion system for brevity, see \cite[Part I, Section 2]{AKO11}.
	\end{definition}
	
	\begin{example}
		If $G$ is a finite group and $S$ is a Sylow $p$-subgroup of $G$, we obtain a saturated fusion system $\calF_S(G)$ on $S$ with morphism sets \[\Hom_G(P,Q) := \{\varphi \in \Hom(P,Q)\mid \varphi = c_g \text{ for some } g \in G\}.\]
	\end{example}
	
	Linckelmann--Webb in an unpublished manuscript first considered the idea of linking fusion systems to biset Burnside rings. The following definition is credited to them. 
	
	\begin{definition}
		Let $S$ be a $p$-group. A \emph{characteristic element} of $\calF$ is an element $X$ in the \emph{$p$-localized biset Burnside ring} $A(S,S)_{(p)}:= \bbZ_{(p)}\otimes_\bbZ A(S,S)$ satisfying the following properties, referred to in \cite{RS13} as the \emph{Linckelmann--Webb properties}:
		\begin{enumerate}
			\item \emph{$\calF$-generation}: $X$ is a linear combination of the twisted diagonal bifree $(S,S)$-bisets $[P,\varphi]^S_S$ (see \Cref{ex:twisted_diag});
			\item \emph{Right $\calF$-stability}: For all $P \leq S$ and $\varphi \in \Hom_\calF(P,S)$, we have in $A(S,P)_{(p)}$ an equality \[X \circ [P,\varphi]^S_P = X \circ [P,\on{incl}]^S_P.\]
			\item \emph{Left $\calF$-stability}: For all $P\leq S$ and $\varphi \in \Hom_\calF(P,S)$, we have in $A(P,S)_{(p)}$ an equality \[[\varphi(P), \varphi\inv]^P_S \circ X = [P,\id]^P_S \circ X.\]
			\item $\epsilon(X) := |X|/|S|$ (the \emph{augmentation} of $X$) is coprime to $p$.
		\end{enumerate}
		A characteristic element of $\calF$ that is in addition idempotent in $A(S,S)_{(p)}$ is called a \emph{characteristic idempotent}.
	\end{definition}
	
	Ragnarsson first showed that every saturated fusion system has a unique characteristic idempotent \cite{Rag06}, and Puig demonstrated that a fusion system is saturated if and only if it has a characteristic element \cite{Pui09}. These facts were strengthened by Ragnarsson--Stancu as follows. 
	
	\begin{theorem}\cite[Theorem C]{RS13}
		For any finite $p$-group $S$, there is a bijective correspondence between saturated fusion systems on $S$ and bifree idempotents in $A(S,S)_{(p)}$ of augmentation 1 satisfying Frobenius reciprocity (see \cite[Definition 7.1]{RS13}). The bijection sends a saturated fusion system $\calF$ to its characteristic idempotent $\omega_\calF$. 
	\end{theorem}
	
	The structure of characteristic idempotents has been studied in greater detail, see e.g. \cite{BD12, GSY15, Ree16}. 
	
	\subsection{$\calF$-stabilization for $p$-localized biset functors}
	
	Though characteristic idempotents $\omega_\calF \in A(S,S)_{(p)}$ of saturated fusion systems $\calF$ may not necessarily be lifted to elements of $A(S,S)$, they are still useful $p$-locally. The key point is that action by $\omega_\calF$ on a $\bbZ_{(p)}$-linear biset functor $\calF$-stabilizes elements. 
	
	\begin{definition}\cite[Definition 2.2]{RY18}
		Let $\calF$ be a fusion system on a $p$-group $S$ and let $F$ be a $p$-biset functor. We set $F(S)^\calF$ to be the set of elements $x\in F(S)$ satisfying \[T_{[P, \varphi]^S_P}(x) = T_{[P, \on{incl}]^S_P}(x)\] for all $P \leq S$ and $\varphi \in \Hom_\calF(P,S)$. 
	\end{definition}
	
	\begin{remark}\label{def:stable_biset}
		For the bisets $\RO(-)$, $\CF_b(-)$, and $A(-)^\times$, the biset-theoretic $\calF$-stable subgroups are exactly the standard $\calF$-theoretic subgroups $\RO(\calF)$, $\CF_b(\calF)$, and $A(\calF)^\times$. 
		It is routine to verify the following from \Cref{def:class_fns_borel_smith_biset}:
		\begin{itemize}
			\item $\on{RO}(S)^\calF$ is equivalently the subring $\RO(\calF)$ of virtual $\bbF$-stable real representations, i.e., those satisfying $\Res^S_P V \cong c_{\varphi}\circ\Res^S_{\varphi(P)}{} V$ (see \cite[Section 3]{RY18});
			\item For any commutative ring $R$, $\CF(S,R)^\calF$ is equivalently the subgroup $\CF(\calF,R) $ consisting of $\calF$-stable superclass functions, i.e., functions constant on $\calF$-isomorphism classes (see \cite[Section 4]{RY18});
			\item $A(S)^\calF$ is equivalently the subring $A(\calF)$ with $\calF$-stable mark homomorphisms (see \cite{DL09});
			\item The subgroup $\CF_b(S)^\calF$ is simply $\CF_b(\calF) =\CF(\calF) \cap \CF_b(S)$ (see \cite[Section 4]{RY18});
			\item The subgroup $(A(S)^\times)^\calF$, viewed as a subgroup of $\CF(S,\bbF_2)$ via \Cref{prop:identification_of_unit_biset}, is $A(\calF)^\times$ (see \cite{BC20}).
		\end{itemize}
		We note there are other equivalent definitions of $\calF$-stability on the Burnside ring, see \cite[Subsection 3.1]{Ree16}.
	\end{remark}
	
	\begin{notation}
		Given any ($\bbZ$-linear) $p$-biset functor $F$, we set $F_{(p)} := \bbZ_{(p)}\otimes_\bbZ F$. Then $F_{(p)}$ is a $\bbZ_p$-linear $p$-biset functor. 
	\end{notation}
	
	The following proposition of Reeh--Yal\c{c}in is a general form of \cite[Theorem A]{Ree16}, and is critical for our purposes.
	
	\begin{proposition}\cite[Proposition 2.5]{RY18}\label{prop:make_fusion_stable}
		Let $\calF$ be a saturated fusion system on a $p$-group $S$, and let $F$ be a $p$-biset functor. We have \[F(S)^\calF_{(p)} = T_{\omega_\calF}(F(S)_{(p)}).\] Moreover, $T_{\omega_\calF}$ is idempotent and for all $x \in F(S)^\calF_{(p)}$, we have $T_{\omega_\calF}(x) = x$. 
	\end{proposition}
	
	This is readily applied to $\RO(-)_{(p)}$, $\CF_b(-)_{(p)}$, and $A(-)_{(p)}^\times$.

	\section{The main results}
	
	We are now ready to prove our main theorems. The first theorem answers the question posed in \cite[Remark 4.2]{MM26}: is $\sgn_\calF\colon \CF_b(\calF) \to A(\calF)^\times$ surjective? 
	
	\begin{theorem}\label{thm:main}
		Let $\calF$ be a saturated fusion system over a $p$-group $S$. The $\calF$-restricted sign homomorphism \[\sgn_\calF\colon \CF_b(\calF) \to A(\calF)^\times\] and the $\calF$-restricted tom Dieck homomorphism \[\Theta_\calF\colon \RO(\calF) \to A(\calF)^\times\] are well-defined and surjective.
	\end{theorem}
	\begin{proof}
		Well-definedness follows from \cite[Lemma 4.5]{RY18}, which verifies that $\dim_\calF$ is well-defined, and well-definedness of $\sgn_\calF$ follows easily from the characterizations in \Cref{prop:sgn_is_biset_mor} and \Cref{def:stable_biset}. If $p$ is odd, then $A(S)^\times = \{\pm[S/S]\}$ and the result is straightforward, so assume $p = 2$. 
		
		It suffices to show the result for $\Theta_\calF$. First, \cite{Tor84,Y05} show the unrestricted tom Dieck homomorphism $\Theta_S$ is surjective. Choose a $\calF$-stable element $x \in A(S)^\times$, then there exists an virtual real representation $X \in \RO(S)$ with $\Theta_S(X) = x$. 
		
		Consider $X \in \RO(S)_{(p)}$ and $x \in (A(S)^\times_{(p)})^\calF$ respectively; we have a $p$-localized tom Dieck morphism of $\bbZ_{(p)}$-linear $p$-biset functors $\Theta_{(p)}\colon \RO(-)_{(p)} \to A(-)^\times_{(p)}$.  
		By biset functoriality, \[(\Theta_{(p)}\circ T_{\omega_\calF})(X) = (T_{\omega_\calF}\circ \Theta_{(p)})(X) = T_{\omega_\calF}(x) = x,\] where the last equality follows from \Cref{prop:make_fusion_stable} applied to $A(-)^\times_{(p)}$. Set $X' := T_{\omega_\calF}(X) \in \RO(S)_{(p)}$, then $X'$ is $\calF$-stable by \Cref{prop:make_fusion_stable} applied to $\RO(-)_{(p)}$, i.e., $X' \in \on{RO}(S)^\calF_{(p)} = \RO(\calF)_{(p)}$, with equality holding from \Cref{def:stable_biset}. Moreover, $\Theta_{(p)}(X') = x$, but $X'$ may not live in $\RO(\calF) \subset \RO(\calF)_{(p)}$. 
		
		We have \[X' = \sum_{i = 1}^n\frac{1}{a_i}\otimes X_i,\] with each $X_i \in \RO(S)$ and $a_i \in \bbZ_{\on{odd}}$. Let $a = \on{lcm}(a_1,\dots, a_n)$ and set $X'' := a X'$. Then $X'' \in \RO(\calF)$. Because $a$ is odd, we have \[\sgn_{(p)}(\dim_{(p)}(X'')) = \sgn_{(p)}(a\cdot \dim_{(p)}(X')) = \sgn_{(p)}(\dim_{(p)}(X')) = \Theta_{(p)}(X') = x.\] Clearly $\Theta_S \equiv \Theta_{(p)}$ restricted to $\RO(-)\subset \RO(-)_{(p)}$, so we are done. 
	\end{proof}
	
	\begin{remark}\label{rmk:endotriv_overview}
		We turn to the application of \Cref{thm:main}. Let $k$ be a field of prime characteristic $p$. Recall that a \emph{permutation} $kG$-module is the $k$-linearization of a $G$-set, and a \emph{$p$-permutation} $kG$-module is a direct summand of a permutation $kG$-module. The bounded homotopy category of $p$-permutation modules, $\catK(G;k) := \on{K}_b(\mathsf{perm}^\natural(kG))$ has a symmetric tensor product, and therefore, one can ask for its \emph{invertible objects}, the bounded chain complexes $C$ of $p$-permutation modules satisfying \[C^* \otimes_k C \simeq k[0].\] Such complexes are termed \emph{endotrivial complexes} (although this name can be misleading, as \emph{endotrivial modules}, the invertible objects of $\mathsf{stmod}(kG)$, are not necessarily endotrivial complexes) and the collection of isomorphism classes of them forms the \emph{Picard group} $\Pic(\catK(G;k))$. 
		
		These were first studied by the author in \cite{M24a} and completely classified in \cite{M24c}, although numerous open questions surrounding these complexes remain. In fact, for a $p$-group $S$, one has an explicit isomorphism $h\colon \Pic(\catK(S;k)) \xrightarrow{\sim} \CF_b(S)$, although describing an explicit inverse remains an open problem. Endotrivials play a significant role in understanding the tensor-triangular geometry of $\catK(G;k)$, see \cite{BG25, Mil26}.  
		
		As $\catK(G;k)$ is a tensor-triangulated category, one may consider its Grothendieck ring as a triangulated category. This is denoted $O(T(kG))$, and called the \emph{trivial source ring}. Equivalently, $O(T(kG))$ is the split Grothendieck ring of the additive monoidal category $\mathsf{perm}(kG)^\natural$. The image of an complex $C\in \catK(G;k)$ in $O(T(kG))$ is called its \emph{Euler characteristic} or \emph{Lefschetz invariant}, and taking the image of an endotrivial complex in $O(T(kG))$ induces the \emph{Lefschetz homomorphism:} \[\Lambda\colon \Pic(\catK(G;k)) \to O(T(kG)), \quad C \mapsto \sum_{i \in \bbZ} (-1)^i [C_i].\] In \cite{M24a}, we showed that if $p=2$ and $k$ contained a 3rd root of unity, then $\Lambda$ was not surjective for the group $A_4$. On the other hand, in \cite{M24c}, we deduced that $\Lambda$ is surjective when $G$ is a $p$-group and $k$ is any field of characteristic $p$. Finally, in \cite{MM26}, for every odd prime $p$ and field $k$ of characteristic $p$, we found examples of groups $G$ for which $\Lambda$ is not surjective, although we showed $\Lambda$ is surjective for all groups with cyclic $p$-Sylows. However, the case of $\bbF_2$ remained open. 
	\end{remark}
	
	\begin{corollary}\label{cor:thmB}
		Let $G$ be a finite group. The Lefschetz homomorphism \[\Lambda\colon \Pic(\catK(G;\bbF_2)) \to O(T(\bbF_2G)), \quad C \mapsto \sum_{i \in \bbZ} (-1)^i[C_i]\] is surjective. In particular, the following holds: if $k$ is field of characteristic $p$, then the Lefschetz homomorphism is surjective for every finite group if and only if $k = \bbF_2$.
	\end{corollary}
	\begin{proof}
		Surjectivity of $\Lambda$ for $\bbF_2 G$ follows directly from \Cref{thm:main} and the remark after \cite[Theorem 4.1]{MM26}, which verifies that $\Lambda$ is surjective if and only if the sign homomorphism $\CF_b(G)^{\calF_S(G)} \to (A(S)^\times)^{\calF_S(G)}$ is surjective. We note that this fact depends on a description of $O(T(kG))$ due to Boltje--Carman \cite{BC23} which is compatible with the Lefschetz homomorphism.
		
		For the latter statement, it suffices to show that for any other field $k\neq \bbF_2$ of positive characteristic, the Lefschetz homomorphism is not surjective for some finite group $G$. For fields $k$ of odd characteristic, such groups are constructed in \cite[Theorem 5.10]{MM26}, and for fields $k \supsetneq \bbF_2$ of characteristic 2, \cite[Example 7.7]{M24a} verifies the Lefschetz homomorphism is not surjective for $A_4$. 
	\end{proof}
	
	\begin{remark}
		In \cite[Theorem 4.5]{MM26}, we concluded \Cref{cor:thmB} holds for fusion systems $\calF_S(G)$ of finite groups $G$ with $2$-Sylow $S$ for which fusion is controlled by the normalizer $N_G(S)$. Equivalently, \Cref{thm:main} holds for fusion systems $\calF_S(G)$. This was shown as follows: given $G$-stable unit $x \in A(\calF_S(G))^\times$, there exists a $f \in \CF_b(S)$ with $\sgn_S(f) = x$. Then, the trace sum \[\Res^{N_G(S)}_S\on{tr}^{N_G(S)}_{S} (f) := \sum_{g \in [N_G(S)/S]}{}^g f\] is $G$-stable, equivalently $\calF_S(G)$-stable, and satisfies $\sgn_S(f') = x$. 
		
		In fact, modulo scaling, this is exactly what the action of the characteristic idempotent $\omega_{\calF_S(G)}$ does to an element $f \in \CF_b(S)_{(p)}$, and therefore, in the proof of \Cref{thm:main}, we obtain the same element . Indeed, one may verify that the element \[\omega := \frac{1}{[N_G(S):S]}\cdot \sum_{g\in [N_G(S)/S]} [S, c_g]^S_S \in A(S,S)_{(p)}\] is idempotent and satisfies the Linckelmann--Webb conditions. By uniqueness of characteristic idempotents \cite[Theorem C]{RS13}, we have $\omega = \omega_{\calF_S(G)}$. Then, one can check $T_{[S, c_G]}(f) = {}^gf$ in $\CF_b(S)_{(p)}$, therefore \[T_\omega (f) = \frac{1}{[N_G(S):S]} \cdot \sum_{g\in [N_G(S)/S]} {}^gf \in \on{CF}_b(\calF)_{(p)}.\] This $p$-local element in the proof of \Cref{thm:main} is scaled to the original element \[\Res^{N_G(S)}_S \on{tr}^{N_G(S)}_S(f)= [N_G(S):S]T_\omega(f) \in \CF_b(\calF).\]
	\end{remark}
	
	\bibliography{bib}
	\bibliographystyle{alpha}
	
\end{document}